\documentclass{amsart}
\usepackage[margin=1in]{geometry}
\usepackage{amsmath, amsthm, amssymb, color, url}
\usepackage{hyperref}
\hypersetup{
    colorlinks,
    citecolor=blue,
    filecolor=black,
    linkcolor=black,
    urlcolor=black
}
\usepackage[final]{showkeys}
\usepackage[disable]{todonotes}

\providecommand{\url}[1]{\url{#1}}

\newcommand{\CC}{\mathbb{C}}
\newcommand{\RR}{\mathbb{R}}

\newcommand{\Tr}{\text{Tr}}

\newcommand{\eps}{\varepsilon}

\newcommand{\wtilde}{\widetilde}

\newcommand{\GT}{\mathrm{GT}}
\newcommand{\trig}{\text{trig}}

\newcommand{\FF}{\mathcal{F}}

\newcommand{\hpi}{\widehat{\pi}}

\newcommand{\Haar}{\mathsf{Haar}}
\newcommand{\EE}{\mathbb{E}}
\newcommand{\BB}{\mathcal{B}}
\newcommand{\Deltat}{\Delta^\text{trig}}

\newcommand{\wB}{\wtilde{\BB}}
\newcommand{\wF}{\wtilde{\FF}}

\newcommand{\wW}{\wtilde{C}^{\text{Wish}}}
\newcommand{\wBB}{\wtilde{C}^{\text{MVB}}}
\newcommand{\wH}{\wtilde{C}^{\text{HO}}}
\newcommand{\wJ}{\wtilde{C}^{\text{Jac}}}

\theoremstyle{definition}

\newtheorem{theorem}{Theorem}[section]

\newtheorem{corr}[theorem]{Corollary}
\newtheorem{lemma}[theorem]{Lemma}
\newtheorem{prop}[theorem]{Proposition}

\newtheorem*{remark}{Remark}

\numberwithin{equation}{section}

\begin{document}

\title{Matrix models for multilevel Heckman-Opdam and multivariate Bessel measures}
\author{Yi Sun}
\address{Y.S.: Department of Mathematics\\ Columbia University\\ 2990 Broadway\\ New York, NY 10027, USA}
\email{yisun@math.columbia.edu}
\date{\today}

\begin{abstract}
We study multilevel matrix ensembles at general $\beta$ by identifying them with a class of processes defined via the branching rules for multivariate Bessel and Heckman-Opdam hypergeometric functions.  For $\beta = 1, 2$, we express the joint multilevel density of the eigenvalues of a generalized $\beta$-Wishart matrix as a multivariate Bessel ensemble, generalizing a result of Dieker-Warren in \cite{DW}.  In the null case, we prove the conjecture of Borodin-Gorin in \cite{BG} that the joint multilevel density of the $\beta$-Jacobi ensemble is given by a principally specialized Heckman-Opdam measure.
\end{abstract}

\maketitle

\tableofcontents

\section{Introduction}

The purpose of the present work is to provide a link between measures defined via the branching structure of certain multivariate hypergeometric functions at general $\beta$ and the multilevel eigenvalue measures of certain random matrix models at $\beta = 1, 2$.  In particular, we consider multivariate Bessel and Heckman-Opdam processes defined in a way similar to the Schur process of \cite{OR} and the Macdonald process of \cite{BC}.  At $\beta = 1$, we show that the multivariate Bessel process with general specializations is realized by the multilevel eigenvalue density of the generalized $\beta$-Wishart ensemble, generalizing a result of Dieker-Warren in \cite{DW} for $\beta = 2$.  At $\beta = 1, 2$, we prove a conjecture of Borodin-Gorin in \cite{BG} that the Heckman-Opdam process at shifted principal specializations is realized by the multilevel eigenvalue density of the $\beta$-Jacobi ensemble.  

The motivation for our work stems from the work of Borodin-Gorin in \cite{BG}, where they use degenerations of techniques from Macdonald processes to show that a rescaling of the Heckman-Opdam process converges to a $2$ dimensional Gaussian free field.  Combined with our identification of the Heckman-Opdam and $\beta$-Jacobi ensembles at $\beta = 1, 2$, this reveals a Gaussian free field structure in the eigenvalues of random matrices, as was first shown for Wigner random matrices in \cite{Bor}.  A special case of the generalized $\beta$-Wishart case is the real spiked covariance model, which admits statistical applications (see \cite{Joh, OMH}) and exhibits the Baik-Ben~Arous-P\'ech\'e phase transition for the largest eigenvalue (see \cite{BBP, Mo, BV}), and it would be interesting to apply this work to analyze it from the perspective of integrable probability.

In the remainder of this introduction, we state our results more precisely and provide additional motivation and background.  For convenience, all notations will be redefined in later sections.

\subsection{Results for generalized $\beta$-Wishart ensembles}

For $\beta > 0$, let $\BB^{n, m}_\beta(\lambda, s)$ and $\wB^{n, m}_\beta(\lambda, s)$ denote the multivariate Bessel and dual multivariate Bessel functions with parameter $\beta$, defined in detail in Section \ref{sec:mvb-def}.  Fix $\theta = \beta/2$.  The multivariate Bessel ensemble with parameters $\{\pi_i\}_{1 \leq i \leq n}$ and $\{\hpi_j\}_{j \geq 1}$ is the process on $\{\mu^m_i\}_{1 \leq i \leq \min\{n, m\}}$ whose distribution on the first $m$ levels is supported on 
\[
\mu^1 \prec \cdots \prec \mu^m, \qquad \mu^l_i \in [0, \infty)
\]
with joint density proportional to 
\begin{multline*}
\Delta(\mu^m)^{\theta} e^{-\theta\sum_{l = 1}^m \hpi_l (|\mu^l| - |\mu^{l-1}|)} \prod_{i = 1}^{\min\{m, n\}} (\mu^m_i)^{\theta (n - \min\{n, m\})} \\ \prod_{i = 1}^{\min\{n, m\}} \frac{(\mu^{\min\{n, m\}}_i)^{\theta - 1}}{(\mu^m_i)^{\theta - 1}}\prod_{l = 1}^{m - 1} \frac{\Delta(\mu^l, \mu^{l+1})^{\theta - 1}}{\Delta(\mu^l)^{\theta - 1}\Delta(\mu^{l+1})^{\theta - 1}} \wB_\beta^{\min\{n, m\}, n}(\mu^m, - \theta\pi)
\end{multline*}
where for sets of variables $\mu$ and $\lambda$ we define $\Delta(\mu) = \prod_{i< j} (\mu_i - \mu_j)$ and $\Delta(\mu, \lambda) = \prod_{i, j} |\mu_i - \lambda_j|$.  The marginal density on level $m$ is proportional to 
\[
\Delta(\mu^m)^{2\theta} \prod_{i = 1}^{\min\{m, n\}} (\mu^m_i)^{\theta (\max\{n, m\} - \min\{n, m\})} \BB_\beta^{\min\{n, m\}, m}(\mu^m, -\theta \hpi) \wB_\beta^{\min\{n, m\}, n}(\mu^m, - \theta \pi).
\]
At $\beta = 2$, a multilevel matrix model was proposed in \cite{BP} and proven in \cite{DW} for the multivariate Bessel ensemble.  Our first main result is a construction of such a matrix model at $\beta = 1$.  Fix some $n \geq 1$.  Let $(A_{ij})$ be an infinite matrix of independent real Gaussian random variables with mean $0$ and variance $(\pi_i + \hpi_j)^{-1}$, and let $A_m$ be its top $m \times n$ corner.  At $\beta = 1$, define the generalized $\beta$-Wishart process to be the sequence of eigenvalues of the $n \times n$ symmetric positive semi-definite matrices
\[
M_m := A_m^T A_m.
\]
Let $\{\mu_i^m\}_{1 \leq i \leq \min\{m, n\}}$ be the non-zero eigenvalues of $M_m$ and denote their joint distribution by $P^{\pi, \hpi}$.  The following two results show that the generalized $\beta$-Wishart ensemble gives a matrix model for the multivariate Bessel process at $\beta = 1$; they generalize the result of \cite{BP, DW} from the $\beta = 2$ case.

{\renewcommand{\thetheorem}{\ref{thm:beta1-wish}}
\begin{theorem}
For $\beta = 1$, the eigenvalues $\{\mu^m_i\}_{m \geq 1}$ of the multilevel generalized $\beta$-Wishart ensemble with parameters $(\pi, \hpi)$ form a Markov chain with transition kernel
\begin{multline*}
Q_{m - 1, m}^{\pi, \hpi}(\mu^{m-1}, d\mu^m) = \frac{\prod_{i = 1}^n (\pi_i + \hpi_m)^{\frac{1}{2}}}{\Gamma(m/2) \Gamma(1/2)^m} e^{-\frac{1}{2}\hpi_m(|\mu^m| - |\mu^{m-1}|)}\\
 \frac{\prod_{i = 1}^{\min\{m, n\}} (\mu^m_i)^{\frac{n - \min\{m, n\} - 1_{m \leq n}}{2}}}{\prod_{i = 1}^{\min\{m - 1, n\}} (\mu^{m-1}_i)^{\frac{n - \min\{m - 1, n\} - 1_{m \leq n}}{2}}} \frac{h^1_\pi(\mu^m)}{h^1_\pi(\mu^{m-1})}\Delta(\mu^m) \Delta(\mu^m, \mu^{m - 1})^{-1/2} 1_{\mu^{m-1} \prec \mu^m} d\mu^m,
\end{multline*}
where $h^1_\pi(\mu)$ denotes the real HCIZ integral
\[
h^1_\pi(\mu) := \int_{V \in O(n)} e^{-\frac{1}{2} V \pi V^T \mu} d\Haar_V
\]
and $d\Haar_V$ denotes Haar measure on the orthogonal group.
\end{theorem}
\addtocounter{theorem}{-1}}

{\renewcommand{\thetheorem}{\ref{corr:beta1-mb-ml}}
\begin{corr}
For $\beta = 1$, the eigenvalues $\mu^1 \prec \cdots \prec \mu^m$ of the multilevel $\beta$-Wishart ensemble with parameters $(\pi, \hpi)$ have the law of the multivariate Bessel ensemble with parameters $(\pi, \hpi)$.
\end{corr}
\addtocounter{theorem}{-1}}

\subsection{Results for $\beta$-Jacobi ensembles}

For $\beta > 0$, let $\FF_\beta^{n, m}(\lambda, s)$ and $\wF_\beta^{n, m}(\lambda, s)$ denote the Heckman-Opdam and dual Heckman-Opdam hypergeometric functions, defined in detail in Section \ref{sec:ho-def}.  The Heckman-Opdam ensemble with parameters $\{\pi_i\}_{1 \leq i \leq n}$ and $\{\hpi_i\}_{i \geq 1}$ is the probability measure on $\{\mu^l_i\}_{1 \leq i \leq \min\{n, l\}, 1\leq l \leq m}$ whose joint distribution on the first $m$ levels is supported on 
\[
\mu^1 \prec \cdots \prec \mu^m \qquad \mu^l_i \in [0, \infty)
\]
with density proportional to
\begin{multline*} 
\Deltat(\mu^m)^{\theta} \prod_{i = 1}^{\min\{m, n\}} (1 - e^{-\mu_i^m})^{\theta(n - \min\{m, n\})}
 \prod_{i = 1}^{\min\{m, n\}} \frac{(1 - e^{-\mu_i^{\min\{m, n\}}})^{\theta - 1}}{(1 - e^{-\mu_i^m})^{\theta - 1}}\\ e^{-\theta\sum_{l = 1}^m \hpi_l (|\mu^l| - |\mu^{l-1}|)} \prod_{l = 1}^{m - 1} \frac{\Delta(e^{-\mu^l}, e^{-\mu^{l+1}})^{\theta - 1}}{\Delta(e^{-\mu^l})^{\theta - 1} \Delta(e^{-\mu^{l + 1}})^{\theta - 1}} e^{(\theta - 1) |\mu^l|} \wF_\beta^{\min\{m, n\}, n}(\mu^m, - \theta\pi).
\end{multline*}
Its marginal density on level $m$ is proportional to 
\[
\prod_{i = 1}^{\min\{m, n\}} (1 - e^{-\mu^m_i})^{\theta (\max\{m, n\} - \min\{m, n\})} \\ \Deltat(\mu^m)^{2\theta} \FF_\beta^{\min\{m, n\}, m}(\mu^m, -\theta\hpi) \wF_\beta^{\min\{m, n\}, n}(\mu^m, - \theta\pi).
\]
Our second main result is an identification of the Heckman-Opdam ensemble at the principal specialization
\[
\pi = (A + m - 1, A + m - 2, \ldots, A + m - n) \text{ and } \hpi = (0, 1, \ldots)
\]
with the density of transformed eigenvalues of the $\beta$-Jacobi ensemble.  The $\beta$-Jacobi ensemble, defined in detail in Section \ref{sec:bj-def}, is constructed as follows. Let $X$ and $Y$ be infinite matrices of independent standard Gaussian random variables over $\RR$ for $\beta = 1$ and $\CC$ for $\beta = 2$.  Choose $A \geq n$ and $m \leq n$, and let $X^{An}$ and $Y^{mn}$ denote the left $A \times n$ and $m \times n$ corners of $X$ and $Y$.  Then the matrix
\[
(X^{An})^* X^{An} ((X^{An})^* X^{An} + (Y^{mn})^* Y^{mn})^{-1}
\]
has $m = \min\{A, m, n\}$ non-zero eigenvalues $\lambda_1^m, \ldots, \lambda_m^m$ in $(0, 1)$.  The $\beta$-Jacobi ensemble is the joint distribution of the eigenvalues $\{\lambda^m_i\}_{1 \leq i \leq m}$ for $1\leq m \leq n$.  The results below explicitly compute its joint density and match it with that of a principally specialized Heckman-Opdam ensemble for $\beta = 1, 2$; together, they resolve a conjecture of Borodin-Gorin in \cite[Section 1.5]{BG}.

{\renewcommand{\thetheorem}{\ref{thm:bg-ver}}
\begin{theorem} 
For $\beta = 2$, the eigenvalues $\{\lambda^l_i\}$ of the first $m$ levels of the $\beta$-Jacobi ensemble with parameters $(A, n)$ are supported on interlacing sequences 
\[
\lambda^1 \prec \cdots \prec \lambda^m \qquad \lambda^l_i \in [0, 1]
\]
with joint density given by
\[
\wJ_{1, m, n} \Delta(\lambda^m) \prod_{i = 1}^m (\lambda^m_i)^{A + m - n - 1} (1 - \lambda^m_i)^{n - m}\prod_{l = 1}^{m - 1} \prod_{i = 1}^l (\lambda^l_i)^{-2}
\]
for a normalization constant $\wJ_{1, m, n}$.
\end{theorem}
\addtocounter{theorem}{-1}}

{\renewcommand{\thetheorem}{\ref{corr:bg-ver-ho}}
\begin{corr} 
For $\beta = 2$, the transformation $\mu^l_i := - \log \lambda^l_i$ of the eigenvalues of the multilevel $\beta$-Jacobi ensemble have the law of the Heckman-Opdam ensemble with parameters $\pi = (A - n + 1, \ldots, A)$ and $\hpi = (0, 1, \ldots)$.
\end{corr}
\addtocounter{theorem}{-1}}

{\renewcommand{\thetheorem}{\ref{thm:bg-ver-beta1}}
\begin{theorem}
For $\beta = 1$, the eigenvalues $\{\lambda^l_i\}$ of the first $m$ levels of the $\beta$-Jacobi ensemble with parameters $(A, n)$ are supported on interlacing sequences
\[
\lambda^1 \prec \cdots \prec \lambda^m \qquad \lambda^l_i \in [0, 1]
\]
with joint density given by
\[
\wJ_{1/2, m, n} \Delta(\lambda^m) \prod_{i = 1}^m (\lambda^m_i)^{\frac{A + m - n - 4}{2}} (1 - \lambda^m_i)^{\frac{n - m + 1}{2}} \prod_{l = 1}^{m - 1} \prod_{i = 1}^l (\lambda^l_i)^{-1} \Delta(\lambda^l) \Delta(\lambda^l, \lambda^{l + 1})^{-1/2}
\]
for a normalization constant $\wJ_{1/2, m, n}$.
\end{theorem}
\addtocounter{theorem}{-1}}

{\renewcommand{\thetheorem}{\ref{corr:bg-ver-ho-beta1}}
\begin{corr}
For $\beta = 1$, the transformation $\mu^l_i := - \log \lambda^l_i$ of the eigenvalues of the multilevel $\beta$-Jacobi ensemble have the law of the Heckman-Opdam ensemble with parameters $\pi = (A - n + 1, \ldots, A)$ and $\hpi = (0, 1, \ldots)$.
\end{corr}
\addtocounter{theorem}{-1}}

\subsection{Relation to the literature}

In \cite{BP}, Borodin-P\'ech\'e introduced the $\beta = 2$ generalized Wishart ensemble and conjectured that its joint eigenvalue structure is described by a $\beta = 2$ multivariate Bessel measure with corresponding specializations.  This conjecture was proven for special values of parameters in \cite{FN} and in general in \cite{DW}.  Our Theorem \ref{thm:beta1-wish} and Corollary \ref{corr:beta1-mb-ml} generalize these results to the $\beta = 1$ case.  This case is of particular interest because it contains the statistically important case of real sample covariance matrices under Johnstone's spiked covariance model; we refer the reader to \cite{Joh, OMH} for some examples of such applications.

In \cite{BG}, Borodin-Gorin identify the single-level marginal principally specialized Heckman-Opdam ensemble with the single level density of the $\beta$-Jacobi ensemble.  They then proved that a rescaling of the multilevel Heckman-Opdam process (termed the $\beta$-Jacobi corners process in \cite{BG}) converges to the Gaussian free field in a large $N$ limit.  Borodin-Gorin conjectured that for $\beta = 1, 2, 4$, the principally specialized Heckman-Opdam ensemble corresponded to a multilevel eigenvalue process for real, complex, and symplectic random Jacobi random matrices, respectively.  Our Theorems \ref{thm:bg-ver} and \ref{thm:bg-ver-beta1} prove this for $\beta = 1, 2$, providing an interpretation of the probabilistic limit theorems of \cite{BG} in terms of random matrix ensembles.

\begin{remark}
A different approach to this identification may be obtained via the random co-rank $1$ projections discussed in \cite{FR}, though we do not pursue this approach further here.
\end{remark}

\subsection{Outline of method and organization}

The remainder of this paper is organized as follows.  In Section \ref{sec:mvb-ho}, we fix our notations for Heckman-Opdam hypergeometric functions and multivariate Bessel functions and derive some identities for them as degenerations of the corresponding identities for Macdonald polynomials.  In Section \ref{sec:mvb-wish}, we define the multivariate Bessel ensemble and use the generalized $\beta$-Wishart ensemble to give a multilevel matrix model for it at $\beta = 1, 2$.  In Section \ref{sec:ho-jac}, we define the Heckman-Opdam ensemble and prove the conjecture of Borodin-Gorin that its principal specialization has matrix model given by the $\beta$-Jacobi ensemble at $\beta = 1, 2$.  In Appendix \ref{sec:asymp}, we collect some elementary computations of limits of different special functions which appear in our limit transitions.

\subsection{Acknowledgements}

The author thanks A. Borodin and V. Gorin for bringing their conjecture to his attention and A. Borodin, P. Etingof, V. Gorin, and E. Rains for helpful discussions. Y.~S. was supported by a NSF Graduate Research Fellowship (NSF Grant \#1122374) and a Junior Fellow award from the Simons Foundation.

\section{Multivariate Bessel and Heckman-Opdam functions} \label{sec:mvb-ho}

In this section, we fix our notations on multivariate Bessel functions, Heckman-Opdam hypergeometric functions, and Macdonald polynomials.  We then describe scaling limits which transform Macdonald polynomials to Heckman-Opdam hypergeometric functions and then multivariate Bessel functions.  Finally, we take limits of the Cauchy identity for Macdonald polynomials to prove Cauchy identities for multivariate Bessel and Heckman-Opdam hypergeometric functions.

\subsection{Notations}

Throughout this paper we denote the rational and trigonometric Vandermonde determinants by
\[
\Delta(\lambda) = \prod_{i < j} (\lambda_i - \lambda_j) \qquad \text{ and } \qquad
\Delta^\trig(\lambda) = \prod_{i < j} \left(e^{\frac{\lambda_i - \lambda_j}{2}} - e^{\frac{\lambda_j - \lambda_i}{2}}\right).
\]
We also use the quantity
\[
\Delta(\mu, \lambda) := \prod_{i, j} |\mu_i - \lambda_j|.
\]
Notice that $\Delta^\trig(\lambda) = e^{\frac{n - 1}{2} |\lambda|} \Delta(e^{-\lambda})$.  For a fixed $\beta > 0$, define $\theta := \frac{\beta}{2}$.  For $\lambda_1 \geq \cdots \geq \lambda_n \in \RR^n$, define the Gelfand-Tsetlin polytope to be 
\[
\GT_\lambda := \{(\mu^l_i)_{1 \leq i \leq l, 1 \leq l < n} \mid \mu^{l+1}_i \geq \mu^l_i \geq \mu^{l+1}_{i+1}, \mu^n_i = \lambda_i\}.
\]
A point $\{\mu^l_i\}$ in $\GT_\lambda$ is called a Gelfand-Tsetlin pattern subordinate to $\lambda$.  For $\mu = (\mu_1 \geq \cdots \geq \mu_{n - 1})$ and $\lambda = (\lambda_1 \geq \cdots \geq \lambda_n)$, we write $\mu \prec \lambda$ to denote that $\mu$ and $\lambda$ interlace, meaning that
\[
\lambda_1 \geq \mu_1 \geq \cdots \geq \mu_{n - 1} \geq \lambda_n.
\]

\subsection{Macdonald polynomials}

We recall some identities for Macdonald polynomials; we refer the reader to the book of \cite{Mac} for a complete treatment.  We will take scaling limits of these to obtain facts on Heckman-Opdam hypergeometric functions and multivariate Bessel functions.  Let $P_\lambda(x; q, t)$ and $Q_\lambda(x; q, t)$ denote the Macdonald and dual Macdonald polynomials.  Recall that
\[
Q_\lambda(x; q, t) = b_\lambda(q, t) P_\lambda(q, t),
\]
where $b_\lambda(q, t)$ is defined by taking $a(s) = \lambda_i - j$ and $l(s) = \lambda_j' - i$ for $s = (i, j)$ and setting
\begin{equation} \label{eq:mac-norm-val}
b_\lambda(q, t) := \langle P_\lambda, P_\lambda \rangle^{-1} = \prod_{s \in \lambda} \frac{1 - q^{a(s)} t^{l(s) + 1}}{1 - q^{a(s) + 1} t^{l(s)}} = \prod_{l = 1}^m \prod_{i = 0}^{l - 1} \frac{(t^{i + 1} q^{\lambda_{l - i} - \lambda_l}; q)_{\lambda_l - \lambda_{l + 1}}}{(t^i q^{\lambda_{l - i} - \lambda_l + 1}; q)_{\lambda_l - \lambda_{l + 1}}}.
\end{equation}
The Macdonald polynomials satisfy a Cauchy identity, evaluation identity, and branching rule.

\begin{prop}[Cauchy identity] \label{prop:mac-cauchy}
For any $x_1, \ldots, x_m$ and $y_1, \ldots, y_n$, we have 
\[
\sum_{\ell(\lambda) \leq \min\{m, n\}} P_\lambda(x_1, \ldots, x_m; q, t) Q_\lambda(y_1, \ldots, y_n; q, t) = \prod_{i = 1}^n \prod_{j = 1}^m \frac{(t x_i y_j; q)}{(x_i y_j; q)}.
\]
\end{prop}

\begin{prop}[Evaluation identity] \label{prop:mac-eval}
For any $\lambda$ and $m$ with $n = \ell(\lambda) \leq m$, we have 
\[
P_\lambda(1, t, \ldots, t^{m - 1}; q, t) = t^{\sum_{i = 1}^n (i - 1)\lambda_i} \prod_{1 \leq i < j \leq n} \frac{(q^{\lambda_i - \lambda_j} t^{j - i}; q)}{(q^{\lambda_i - \lambda_j} t^{j - i + 1}; q)} \frac{(t^{j - i + 1}; q)}{(t^{j - i}; q)} \prod_{i = 1}^n \prod_{j = n + 1}^m \frac{(q^{\lambda_i} t^{j- i}; q)}{(q^{\lambda_i} t^{j - i + 1}; q)} \frac{(t^{j - i + 1}; q)}{(t^{j - i}; q)}.
\]
\end{prop}

\begin{prop}[Branching rule] \label{prop:mac-branch}
For $\lambda = (\lambda_1 \geq \cdots \geq \lambda_n, 0, \ldots, 0)$, we have
\[
P_\lambda(x_1, \ldots, x_m; q, t) = \sum_{\mu \prec \lambda} \psi^m_{\lambda/\mu}(q, t) P_\mu(x_1, \ldots, x_{m - 1}; q, t) x_m^{|\lambda| - |\mu|}
\]
for the branching coefficient defined for $\ell(\lambda) = m$ in terms of $f(u) := \frac{(tu; q)}{(qu; q)}$ by
\[
\psi^m_{\lambda/\mu}(q, t) = \prod_{1 \leq i \leq j \leq m - 1}  \frac{f(q^{\mu_i - \mu_j} t^{j - i}) f(q^{\lambda_i - \lambda_{j + 1}} t^{j - i})}{f(q^{\mu_i - \lambda_{j + 1}}t^{j - i}) f(q^{\lambda_i - \mu_j} t^{j - i})}.
\]
\end{prop}

\begin{corr}[Truncated branching rule] \label{corr:mac-branch-trunc}
For $\ell(\lambda) = n < m$, we have for $f(u) := \frac{(tu; q)}{(qu; q)}$ that
\[
\psi^m_{\lambda/\mu}(q, t) = \prod_{1 \leq i < j \leq n} \frac{f(q^{\mu_i - \mu_j} t^{j - i}) f(q^{\lambda_i - \lambda_{j}} t^{j - i - 1})}{f(q^{\lambda_i - \mu_j} t^{j - i}) f(q^{\mu_i - \lambda_{j}}t^{j - i - 1})} \prod_{i = 1}^n \frac{f(1)}{f(q^{\lambda_i - \mu_i})} \prod_{i = 1}^n \frac{f(q^{\lambda_i} t^{n - i})}{f(q^{\mu_i} t^{n - i})},
\]
where $\lambda = (\lambda_1 \geq \cdots \geq \lambda_n, 0, \ldots, 0)$ and $\mu = (\mu_1 \geq \cdots \geq \mu_n, 0, \ldots, 0)$.
\end{corr}

We compute now some quasi-classical limits of Macdonald polynomials at both a general specialization and the principal specialization.  These results overlap with those of \cite[Section 6]{BG}, but we include them here for the reader's convenience.

\begin{lemma} \label{lem:mac-eval-lim}
For $q = e^{-\eps}$ and $t = e^{-\theta\eps}$, we have 
\[
\lim_{\eps \to 0} \eps^{\theta ((m - n)n + n(n - 1)/2)} P_{\eps^{-1}\lambda}(1, t, \ldots, t^{m - 1}; q, t) =\frac{\Gamma(\theta)^n}{\Gamma(m\theta) \cdots \Gamma((m - n + 1)\theta)} \Delta(e^{-\lambda})^\theta \prod_{i = 1}^n (1 - e^{-\lambda_i})^{\theta (m - n)}.
\]
\end{lemma}
\begin{proof}
Applying Lemmas \ref{lem:qpoch-asymp} and \ref{lem:ratio-qpoch} in Proposition \ref{prop:mac-eval}, we find that 
\begin{align*}
\lim_{\eps \to 0} &\eps^{\theta ((m - n)n + n(n - 1)/2)} P_{\eps^{-1}\lambda}(1, t, \ldots, t^{m - 1}; q, t)\\ 
&= e^{-\theta\sum_{i = 1}^n (i - 1)\lambda_i} \prod_{1 \leq i < j \leq n} (1 - e^{\lambda_j - \lambda_i})^\theta \frac{\Gamma((j - i)\theta)}{\Gamma((j - i + 1)\theta)} \prod_{i = 1}^n \prod_{j = n + 1}^m (1 - e^{-\lambda_i})^\theta \frac{\Gamma((j - i)\theta)}{\Gamma((j - i + 1)\theta)}\\
&= \frac{\Gamma(\theta)^n}{\Gamma(m\theta) \cdots \Gamma((m - n + 1)\theta)} \Delta(e^{-\lambda})^\theta \prod_{i = 1}^n (1 - e^{-\lambda_i})^{\theta (m - n)}. \qedhere
\end{align*}
\end{proof}

\begin{lemma} \label{lem:mac-branch-lim}
For $q = e^{-\eps}$ and $t = e^{-\theta \eps}$, if $\ell(\lambda) = m$, we have 
\[
\lim_{\eps \to 0} \eps^{(\theta - 1)n} \psi^m_{\eps^{-1}\lambda/\eps^{-1}\mu}(q, t) =  \Gamma(\theta)^{1 - m} \frac{\Delta(e^{-\mu}, e^{-\lambda})^{\theta - 1}}{\Delta(e^{-\mu})^{\theta - 1} \Delta(e^{-\lambda})^{\theta - 1}} e^{(\theta - 1)|\mu|}.
\]
\end{lemma}
\begin{proof}
Taking the limit of Proposition \ref{prop:mac-branch}, we find by applying Lemma \ref{lem:f-lim} that 
\begin{align*}
\lim_{\eps \to 0} \eps^{(\theta - 1)(m - 1)} \psi^m_{\eps^{-1}\lambda/\eps^{-1}\mu}(q, t) &= \Gamma(\theta)^{1 - m}\!\!\!\!\prod_{1 \leq i < j \leq m - 1} (1 - e^{\mu_j - \mu_i})^{1 - \theta}\!\!\!\!\!\!\! \prod_{1 \leq i \leq j \leq m - 1}\!\! \frac{(1 - e^{\lambda_{j + 1} - \lambda_j})^{1 - \theta}}{(1 - e^{\lambda_{j + 1} - \mu_i})^{1 - \theta} (1 - e^{\mu_j - \lambda_i})^{1 - \theta}}\\
&= \Gamma(\theta)^{1 - m} \frac{\Delta(e^{-\mu}, e^{-\lambda})^{\theta - 1}}{\Delta(e^{-\mu})^{\theta - 1} \Delta(e^{-\lambda})^{\theta - 1}} e^{(\theta - 1)|\mu|}. \qedhere
\end{align*}
\end{proof}

\begin{lemma} \label{lem:mac-tbranch-lim}
For $q = e^{-\eps}$ and $t = e^{-\theta \eps}$, if $\ell(\lambda) = n < m$, we have 
\[
\lim_{\eps \to 0} \eps^{(\theta - 1)n} \psi^m_{\eps^{-1}\lambda/\eps^{-1}\mu}(q, t) = \Gamma(\theta)^{-n} \frac{\Delta(e^{-\mu}, e^{-\lambda})^{\theta - 1}}{\Delta(e^{-\mu})^{\theta - 1} \Delta(e^{-\lambda})^{\theta - 1}} \prod_{i = 1}^n \frac{(1 - e^{-\mu_i})^{\theta - 1}}{(1 - e^{-\lambda_i})^{\theta - 1}} e^{(\theta - 1)|\mu|}. 
\]
\end{lemma}
\begin{proof}
Taking the limit of Corollary \ref{corr:mac-branch-trunc}, we find by applying Lemma \ref{lem:f-lim} that 
\begin{align*}
\lim_{\eps \to 0} \eps^{(\theta - 1)n} \psi^m_{\eps^{-1}\lambda/\eps^{-1}\mu}(q, t) &= \Gamma(\theta)^{-n}\!\!\!\!\! \prod_{1 \leq i < j \leq n}\! \frac{(1 - e^{\mu_j - \lambda_i})^{\theta - 1} (1 - e^{\lambda_j - \mu_i})^{\theta - 1}}{(1 - e^{\mu_j - \mu_i})^{\theta - 1} (1 - e^{\lambda_j - \lambda_i})^{\theta - 1}} \prod_{i = 1}^n  \frac{(1 - e^{\mu_i - \lambda_i})^{\theta - 1}(1 - e^{-\mu_i})^{\theta - 1}}{(1 - e^{-\lambda_i})^{\theta - 1}} \\
&= \Gamma(\theta)^{-n} \frac{\Delta(e^{-\mu}, e^{-\lambda})^{\theta - 1}}{\Delta(e^{-\mu})^{\theta - 1} \Delta(e^{-\lambda})^{\theta - 1}} \prod_{i = 1}^n \frac{(1 - e^{-\mu_i})^{\theta - 1}}{(1 - e^{-\lambda_i})^{\theta - 1}} e^{(\theta - 1)|\mu|}. \qedhere
\end{align*}
\end{proof}

\subsection{Definition of the Heckman-Opdam hypergeometric function} \label{sec:ho-def}

For $s = (s_1, \ldots, s_m)$ and $\lambda = (\lambda_1, \ldots, \lambda_n, 0, \ldots, 0)$, define the integral formula
\begin{multline*}
\Phi^{n, m}_{\theta}(\lambda, s) = \Gamma(\theta)^{-n(m - n) - n(n-1)/2} \int_{\mu \in \GT_\lambda}\!\!\!\!\!\!\!\!\!\! e^{\sum_{l = 1}^{m} s_l (|\mu^{l}| - |\mu^{l - 1}|)}\\ \prod_{i = 1}^n \frac{(1 - e^{-\mu^n_i})^{\theta - 1}}{(1 - e^{-\lambda_i})^{\theta - 1}}\prod_{l = 1}^{m - 1} \frac{\Delta(e^{-\mu^l}, e^{-\mu^{l + 1}})^{\theta - 1}}{\Delta(e^{-\mu^l})^{\theta - 1} \Delta(e^{-\mu^{l + 1}})^{\theta - 1}} e^{(\theta - 1)|\mu^l|} \prod_{l = 1}^{m - 1} \prod_{i = 1}^{\min\{l, n\}} d\mu^l_i,
\end{multline*}
where the integral is over a space of dimension $(m - n)n + \frac{n(n - 1)}{2}$.  The Heckman-Opdam hypergeometric function is defined by 
\begin{equation} \label{eq:ho-def}
\FF_{\beta}^{n, m}(\lambda, s) := \frac{\Gamma(m\theta) \cdots \Gamma((m - n + 1)\theta)}{\Gamma(\theta)^{n}} \frac{\Phi^{n, m}_{\theta}(\lambda, s)}{\Delta^\trig(\lambda)^\theta \prod_{i = 1}^n (1 - e^{-\lambda_i})^{\theta(m - n)}}.
\end{equation}
Define the conjugate Heckman-Opdam hypergeometric function by 
\begin{equation} \label{eq:ho-conj-def}
\wF_\beta^{n, m}(\lambda, s) := \Gamma(\theta)^{-n} \prod_{i = 1}^n (1 - e^{-\lambda_i})^{\theta - 1} \FF_\beta^{n, m}(\lambda, s).
\end{equation}

\begin{remark}
The integral formula for the Heckman-Opdam hypergeometric function which appears here was first given by Borodin-Gorin in \cite{BG}, though our choice of normalization differs from that of \cite{BG}.  We refer the reader to \cite{HS} for an exposition of the classical theory of these functions and to \cite{Sun:ho} for an explanation of the connection between the two.
\end{remark}

We now prove several limit formulas relating Macdonald polynomials and Heckman-Opdam hypergeometric functions.

\begin{corr}[{\cite[Propositions 6.2 and 6.4]{BG}}] \label{corr:mac-lim} 
If $\lambda = (\lambda_1, \ldots, \lambda_n, 0, \ldots, 0)$ and $s = (s_1, \ldots, s_m)$ with $\lambda_1 > \cdots > \lambda_n > 0$, then 
\[
\lim_{\eps \to 0} \eps^{\theta(n(m - n) + n (n - 1)/2)} P_{\eps^{-1}\lambda}(e^{\eps s_1}, \ldots, e^{\eps s_m}; e^{-\eps}, e^{-\theta \eps}) = \Phi^{n, m}_\theta(\lambda, s).
\]
\end{corr}
\begin{proof}
Combine Proposition \ref{prop:mac-branch} and Lemmas \ref{lem:mac-branch-lim} and \ref{lem:mac-tbranch-lim}.
\end{proof}

We now extend this scaling to $Q_\lambda(x; q, t)$.

\begin{lemma} \label{lem:b-ho-scaling}
For $\lambda_1 > \cdots > \lambda_n > 0$, we have the scaling limit
\[
\lim_{\eps \to 0} \eps^{n(\theta - 1)} b_{\lfloor \eps^{-1} (\lambda_1, \ldots, \lambda_n, 0, \ldots, 0)\rfloor}(e^{-\eps}, e^{-\theta\eps}) = \Gamma(\theta)^{-n}\prod_{i = 1}^n (1 - e^{-\lambda_i})^{\theta - 1}.
\]
\end{lemma}
\begin{proof}
By Lemma \ref{lem:fin-qpoch-asymp}, for $q = e^{-\eps}$, $t = e^{-\theta \eps}$, and $i > 0$, we have that
\[
\lim_{\eps \to 0} \frac{(t^{i + 1} q^{\lambda_{l - i} - \lambda_l}; q)_{\lambda_l - \lambda_{l + 1}}}{(t^i q^{\lambda_{l - i} - \lambda_l + 1}; q)_{\lambda_l - \lambda_{l + 1}}} = \frac{(1 - e^{\lambda_l - \lambda_{l - i}})^{-\theta + 1}}{(1 - e^{\lambda_{l + 1} - \lambda_{l - i}})^{-\theta + 1}}
\]
and 
\[
\lim_{\eps \to 0} \eps^{\theta - 1} \frac{(t; q)_{\lambda_l - \lambda_{l + 1}}}{(q; q)_{\lambda_l - \lambda_{l + 1}}} = \Gamma(\theta)^{-1} (1 - e^{\lambda_{l + 1} - \lambda_l})^{\theta - 1}.
\]
We conclude that 
\begin{align*}
\lim_{\eps \to 0} \eps^{n(\theta - 1)} b_{\lfloor \eps^{-1} (\lambda_1, \ldots, \lambda_n, 0, \ldots, 0)\rfloor}(e^\eps, e^{\theta\eps}) &= \Gamma(\theta)^{-n} \prod_{l = 1}^n (1 - e^{\lambda_{l + 1} - \lambda_l})^{\theta - 1} \prod_{i = 1}^{l - 1} \frac{(1 - e^{\lambda_l - \lambda_{l - i}})^{-\theta + 1}}{(1 - e^{\lambda_{l + 1} - \lambda_{l - i}})^{-\theta + 1}}\\
& = \Gamma(\theta)^{-n}\prod_{i = 1}^n (1 - e^{-\lambda_i})^{\theta - 1}. \qedhere
\end{align*}
\end{proof}

\begin{corr} \label{corr:mac-q-lim}
If $\lambda = (\lambda_1, \ldots, \lambda_n, 0, \ldots, 0)$ and $s = (s_1, \ldots, s_m)$ with $\lambda_1 > \cdots > \lambda_n > 0$, then 
\[
\lim_{\eps \to 0} \eps^{\theta(n(m - n) + n (n - 1)/2) + n (\theta - 1)} Q_{\eps^{-1}\lambda}(e^{\eps s_1}, \ldots, e^{\eps s_m}; e^{-\eps}, e^{-\theta \eps}) = \Gamma(\theta)^{-n} \prod_{i = 1}^n (1 - e^{-\lambda_i})^{\theta - 1} \Phi^{n, m}_\theta(\lambda, s).
\]
\end{corr}

\begin{corr} \label{corr:mac-qeval-lim}
For $q = e^{-\eps}$ and $t = e^{-\theta\eps}$, we have 
\begin{multline*}
\lim_{\eps \to 0} \eps^{\theta ((m - n)n + n(n - 1)/2) + n (\theta - 1)} Q_{\eps^{-1}\lambda}(1, t, \ldots, t^{m - 1}; q, t) \\= \frac{1}{\Gamma(m\theta) \cdots \Gamma((m - n + 1)\theta)} \Delta(e^{-\lambda})^\theta \prod_{i = 1}^n (1 - e^{-\lambda_i})^{\theta (m - n + 1) - 1}.
\end{multline*}
\end{corr}

Using these limit transitions, we may translate the evaluation and Cauchy identities for Macdonald polynomials to Heckman-Opdam hypergeometric functions.

\begin{corr} \label{corr:mac-ho-scale}
For $n \leq m$, we have the scalings
\[
\lim_{\eps \to 0} \frac{P_{\lfloor \eps^{-1} (\lambda_1, \ldots, \lambda_n, 0, \ldots, 0)\rfloor}(e^{\eps s_1}, \ldots, e^{\eps s_m}; e^{-\eps}, e^{-\theta \eps})}{P_{\lfloor \eps^{-1}(\lambda_1, \ldots, \lambda_n, 0, \ldots, 0)\rfloor}(1, e^{-\theta \eps}, \ldots, e^{-(m - 1)\theta \eps}; e^{-\eps}, e^{-\theta\eps})} = e^{\frac{n - 1}{2} \theta |\lambda|}\FF^{n, m}_\beta(\lambda, s)
\]
and
\[
\lim_{\eps \to 0} \frac{Q_{\lfloor \eps^{-1} (\lambda_1, \ldots, \lambda_n, 0, \ldots, 0)\rfloor}(e^{\eps s_1}, \ldots, e^{\eps s_m}; e^{-\eps}, e^{-\theta \eps})}{P_{\lfloor \eps^{-1}(\lambda_1, \ldots, \lambda_n, 0, \ldots, 0)\rfloor}(1, e^{-\theta \eps}, \ldots, e^{-(m - 1)\theta \eps}; e^{-\eps}, e^{-\theta\eps})} = e^{\frac{n - 1}{2} \theta |\lambda|}\wF^{n, m}_\beta(\lambda, s).
\]
\end{corr}

\begin{prop} \label{prop:ho-cauchy}
For $m \geq n$ and parameters $s = (s_1, \ldots, s_m)$ and $r = (r_1, \ldots, r_n)$, we have
\begin{multline*}
\int_{\lambda_1 \geq \cdots \geq \lambda_n \geq 0} \FF_\beta^{n, m}(\lambda, s) \wF_\beta^{n, n}(\lambda, r) \Delta(e^{-\lambda})^{2\theta} e^{(n - 1)\theta |\lambda|} \prod_{i = 1}^n (1 - e^{-\lambda_i})^{\theta(m - n)}d\lambda\\
 = \frac{\Gamma(m \theta) \cdots \Gamma((m - n + 1)\theta) \Gamma(n\theta) \cdots \Gamma(\theta)}{\Gamma(\theta)^{2n}}\prod_{i = 1}^m \prod_{j = 1}^n \frac{\Gamma(-s_i - r_j)}{\Gamma(\theta - s_i - r_j)}.
\end{multline*}
\end{prop}
\begin{proof}
By applying Proposition \ref{prop:mac-cauchy}, Corollaries \ref{corr:mac-lim} and \ref{corr:mac-q-lim}, and Lemma \ref{lem:ratio-qpoch}.
\end{proof}

\subsection{Definition of the multivariate Bessel function} \label{sec:mvb-def}

For $\lambda = (\lambda_1 > \cdots > \lambda_n)$ and $s = (s_1, \ldots, s_m)$ with $n \leq m$, define the integral formula
\begin{multline} \label{eq:int-rat}
\phi^{n, m}_\theta(\lambda, s) = \Gamma(\theta)^{-\frac{n(n - 1)}{2} - n(m - n)}\int_{\mu \in \GT_\lambda}\!\!\!\!\!\!\!\!\!\! e^{\sum_{l = 1}^m s_l(|\mu^l| - |\mu^{l-1}|)}\! \\\prod_{i = 1}^n \frac{(\mu_i^n)^{\theta - 1}}{\lambda_i^{\theta - 1}} \prod_{l = 1}^{m - 1} \frac{\Delta(\mu^l, \mu^{l + 1})^{\theta - 1}}{\Delta(\mu^l)^{\theta - 1} \Delta(\mu^{l + 1})^{\theta - 1}} \prod_{l = 1}^{m - 1} \prod_{i = 1}^{\min\{l, n\}} d\mu^l_i.
\end{multline}
For $n \leq m$, the multivariate Bessel function is defined for ordered and distinct $\lambda$ by
\begin{equation} \label{eq:mvb-int}
\BB^{n, m}_\beta(\lambda, s) := \frac{\Gamma(m\theta) \cdots \Gamma((m - n + 1)\theta)}{\Gamma(\theta)^n} \frac{\phi_\theta(\lambda, s)}{\Delta(\lambda)^\theta \prod_i \lambda_i^{\theta(m - n)}}
\end{equation}
and is extended to all $\lambda$ by continuation.  Define also the conjugate multivariable Bessel function by 
\[
\wB^{n, m}_\beta(\lambda, s) := \Gamma(\theta)^{-n}\prod_{i = 1}^n \lambda_i^{\theta - 1} \BB^{n, m}_\beta(\lambda, s).
\]

\begin{remark}
This integral form for the multivariable Bessel function first appeared in \cite[Section V]{GK}.  We have adjusted the normalization of $\BB_\beta^{n, m}(\lambda, s)$ from \cite{GK} so that $\BB_\beta^{n, m}(\lambda, 0) = 1$.
\end{remark}

From the integral formula, it is straightforward to show the following identities involving the multivariate Bessel function.  We give also a scaling limit from Heckman-Opdam functions to multivariate Bessel functions which yields a Cauchy identity.

\begin{prop} \label{prop:mvb-prop}
For $m \geq n$, the multivariate Bessel function $\BB_\beta^{n, m}(\lambda, s)$ satisfies
\begin{itemize}
\item $\BB_\beta^{n, m}(0, s) = \BB_\beta^{n, m}(\lambda, 0) = 1$;

\item $\BB_\beta^{n, m}(\lambda, s) = \BB_\beta^{m, m}(s, \lambda)$;

\item $c^{\theta(n(m - n) + n(n-1)/2)}\BB_\beta^{n, m}(\lambda, cs) = \BB_\beta^{n, m}(c \lambda, s)$.
\end{itemize}
\end{prop}

\begin{prop} \label{prop:ho-mvb-scale}
We have the scaling limits
\begin{align*}
\lim_{\eps \to 0} \FF_\beta^{n, m}(\eps \lambda, \eps^{-1} s) &= \BB_\beta^{n, m}(\lambda, s)\\
\lim_{\eps \to 0} \eps^{- (\theta - 1)n} \wF_\beta^{n, m}(\eps \lambda, \eps^{-1} s) &= \wB_\beta^{n, m}(\lambda, s)
\end{align*}
\end{prop}
\begin{proof}
By taking explicit scaling limits in the integral expressions for $\FF_\beta^{n, m}(\eps\lambda, \eps^{-1}s)$ and $\BB_\beta^{n, m}(\lambda, s)$.
\end{proof}

\begin{prop} \label{prop:mvb-cauchy}
For $n \leq m$, $s = (s_1, \ldots, s_m)$, and $r = (r_1, \ldots, r_n)$, we have
\begin{multline*}
\int_{\lambda_1 \geq \cdots \geq \lambda_n \geq 0} \BB_\beta^{n, m}(\lambda, -s) \wB_\beta^{n, n}(\lambda, -r) \Delta(\lambda)^{2\theta} \prod_{i = 1}^{n} \lambda_i^{\theta (m - n)} d\lambda\\ = \frac{\Gamma(m \theta) \cdots \Gamma((m - n + 1)\theta) \Gamma(n\theta) \cdots \Gamma(\theta)}{\Gamma(\theta)^{2n}} \prod_{i = 1}^m \prod_{j = 1}^n (s_i + r_j)^{-\theta}.
\end{multline*}
\end{prop}
\begin{proof}
In Proposition \ref{prop:ho-cauchy}, substitute $(\eps\lambda, \eps^{-1} s, \eps^{-1}r)$ for $(\lambda, s, r)$, multiply both sides by $\eps^{-\theta mn}$, and take the limit as $\eps \to 0$, applying Proposition \ref{prop:ho-mvb-scale} and Lemma \ref{lem:gamma-rat}.
\end{proof}

\subsection{Multivariate Bessel functions and HCIZ integrals} \label{sec:mvb-hciz}

In this section we relate the multivariate Bessel functions at $\beta = 1, 2$ to integrals over Haar measure on the orthogonal and unitary groups, known as HCIZ integrals.  For real sequences $(a_1, \ldots, a_m)$ and $(b_1, \ldots, b_n)$ with $n \leq m$, we define the complex HCIZ integral by 
\[
h^2_a(b) := \int_{U \in U(N)} e^{-\Tr(UaU^*b)} d\Haar_U
\]
and the real HCIZ integral by 
\[
h^1_a(b) := \int_{V \in O(N)} e^{-\Tr(VaV^Tb)/2} d\Haar_V,
\]
where we interpret $a$ and $b$ as diagonal matrices of length $N$ by padding with $0$'s, and $d\Haar_U$ and $d\Haar_V$ are the Haar measure on the spaces of unitary and orthogonal matrices, respectively.  

\begin{remark}
In the complex case, if all $a$ and $b$ are distinct and $n = m$, the explicit expression 
\[
h^2_a(b) = C^{\text{HCIZ}}_m \frac{\det(e^{-a_i b_j})}{\Delta(a) \Delta(b)}
\]
for some constant $C^{\text{HCIZ}}_m$ for the HCIZ integral was given in the original works of \cite{H1, H2, IZ}.
\end{remark}

In the following two lemmas, we apply the results of \cite{Ner} to identify HCIZ integrals with multivariate Bessel functions at $\beta = 1, 2$.  

\begin{lemma} \label{lem:hciz-complex}
If $a_1 > \cdots > a_m$ are all distinct, then we have
\[
h^2_a(b) = \BB_2^{m, m}(a, -b).
\]
\end{lemma}
\begin{proof}
The HCIZ integral may be rephrased as an integral over the orbital measure on the conjugacy class of the matrix with $a_1, \ldots, a_m$ as its diagonal entries.  Restricting the $\theta = 1$ case of \cite[Proposition 1.1]{Ner} to this conjugacy class, we see that the orbital measure pushes forward to a multiple of the Lebesgue measure on $\GT_a$.  Matching the resulting integrals shows that the two sides agree up to a multiplicative factor in $a$.  Noting that $h^2_a(0) = 1 = \BB_2^{m, m}(a, 0)$ completes the proof.
\end{proof}

\begin{lemma} \label{lem:hciz-real}
If $a_1 > \cdots > a_m$ are all distinct, then we have
\[
h_a^1(b) = \BB_1^{m, m}(a, -b/2).
\]
\end{lemma}
\begin{proof}
The proof proceeds along the same lines as that of Lemma \ref{lem:hciz-complex} with \cite[Proposition 1.1]{Ner} applied in the case $\theta = 1/2$.
\end{proof}

\section{The multivariate Bessel ensemble and the generalized $\beta$-Wishart ensemble} \label{sec:mvb-wish}

We define the multivariate Bessel ensemble and generalized $\beta$-Wishart ensemble.  At $\beta = 1, 2$, we prove that the generalized $\beta$-Wishart ensemble gives a matrix model for the multivariate Bessel process. The $\beta = 2$ case was considered in \cite{BP} and \cite{DW}, where the analysis hinged on the HCIZ integral in two ways.  The paper \cite{DW} uses results of \cite{Def} which reduce to an evaluation of this integral, and our reduction to the Laplace transform on the Laguerre ensemble also uses it.

\subsection{Definition of the multivariate Bessel ensemble}

The \textit{multivariate Bessel ensemble} with parameters $\{\pi_i\}_{1 \leq i \leq n}$ and $\{\hpi_j\}_{j \geq 1}$ is the process on $\{\mu^m_i\}_{1 \leq i \leq \min\{n, m\}}$ with joint distribution supported on 
\[
\mu^1 \prec \cdots \prec \mu^m, \qquad \mu^l_i \in [0, \infty)
\]
with $\mu^l$ of length $\min\{l, n\}$ and density given by
\begin{multline} \label{eq:mvb-density}
p^{m, n}_{\pi, \hpi}(\mu^1 \prec \cdots \prec \mu^m) := \wBB_{\theta, m, n} \prod_{j = 1}^m \prod_{i = 1}^n (\hpi_j + \pi_i)^\theta \Delta(\mu^m)^{\theta} e^{-\theta\sum_{l = 1}^m \hpi_l (|\mu^l| - |\mu^{l-1}|)}  \\ \prod_{i = 1}^{\min\{m, n\}} (\mu^m_i)^{\theta (n - \min\{n, m\})}\prod_{i = 1}^{\min\{n, m\}}  \frac{(\mu^{\min\{n, m\}}_i)^{\theta - 1}}{(\mu^m_i)^{\theta - 1}}\prod_{l = 1}^{m - 1} \frac{\Delta(\mu^l, \mu^{l+1})^{\theta - 1}}{\Delta(\mu^l)^{\theta - 1}\Delta(\mu^{l+1})^{\theta - 1}} \wB_\beta^{\min\{n, m\}, n}(\mu^m, - \theta\pi)
\end{multline}
for a normalization constant $\wBB_{\theta, m, n}$.

\begin{prop} \label{prop:mvb-valid}
The multivariate Bessel ensemble with parameters $(\pi, \hpi)$ is a valid probability measure with level $m$ marginal density given by
\begin{multline} \label{eq:mvb-marginal-density}
p^{m, n}_{\pi, \hpi}(\mu^m) := C^{\text{MVB}}_{\theta, m, n} \prod_{j = 1}^m \prod_{i = 1}^n (\hpi_j + \pi_i)^\theta\\ \Delta(\mu^m)^{2\theta} \prod_{i = 1}^{\min\{m, n\}} (\mu^m_i)^{\theta (\max\{n, m\} - \min\{n, m\})} \BB_\beta^{\min\{n, m\}, m}(\mu^m, -\theta \hpi) \wB_\beta^{\min\{n, m\}, n}(\mu^m, - \theta \pi)
\end{multline}
for a normalization constant $C^{\text{MVB}}_{\theta, m, n}$.
\end{prop}
\begin{proof}
First, the density (\ref{eq:mvb-density}) is non-negative by its definition.  By the defining integral formula (\ref{eq:mvb-int}) for the multivariate Bessel function, integration over $\mu^1 \prec \cdots \prec \mu^{m - 1}$ in $\GT_{\mu^m}$ yields the expression (\ref{eq:mvb-marginal-density}) with a non-zero constant of proportionality.  The Cauchy identity (\ref{prop:mvb-cauchy}) for multivariate Bessel functions then ensures that (\ref{eq:mvb-marginal-density}) and therefore (\ref{eq:mvb-density}) has finite total mass and may therefore be normalized to a valid probability measure with the claimed marginal density.
\end{proof}

\subsection{Definition of the generalized $\beta$-Wishart ensemble}

Let $\beta = 1, 2$, and fix two sets of non-negative real parameters $\{\pi_i\}_{1 \leq i \leq n}$ and $\{\hpi_i\}_{i \geq 1}$.  Let $(A_{ij})$ be an infinite matrix of zero-mean Gaussian random variables with variance $(\pi_j + \hpi_i)^{-1}$ which are real if $\beta = 1$ and complex if $\beta = 2$.  Let $A_m := (A_{ij})_{1 \leq i \leq m, 1 \leq j \leq n}$ denote its top $m$ rows.  Following \cite{BP} in the case $\beta = 2$, we define the \textit{generalized $\beta$-Wishart ensemble} with parameters $(\pi, \hpi)$ and level $m$ to be the random matrix $M_m = A_m^* A_m$.  We call the joint distribution of $\{M_m\}_{1 \leq m}$ the \textit{multilevel generalized $\beta$-Wishart ensemble} with parameters $(\pi, \hpi)$.

\begin{remark}
Setting $\pi_j = 1$ and $\hpi_i = 0$ recovers the standard Wishart process.
\end{remark}

Denote the non-zero eigenvalues of $M_m$ by $\{\mu^m_i\}_{1 \leq i \leq \min\{m, n\}}$.  For $\beta = 2$, it was shown by Dieker-Warren in \cite{DW} that the multivariate Bessel ensemble admits a matrix model in terms of a multilevel generalized $\beta$-Wishart ensemble.

\begin{theorem}[{\cite[Theorem 3.1]{DW}}] \label{thm:gen-wish-dens}
For $\beta = 2$, the eigenvalues $\{\mu^m_i\}$ of the multilevel generalized $\beta$-Wishart ensemble with parameters $(\pi, \hpi)$ form a Markov process with transition kernel
\begin{multline*}
Q^{\pi, \hpi}_{m - 1, m}(\mu^{m-1}, d\mu^{m}) = \prod_{i = 1}^n(\pi_i + \hpi_m)\\ \frac{\prod_{i = 1}^{\min\{m, n\}} (\mu^m_i)^{n - \min\{m, n\}}}{\prod_{i = 1}^{\min\{m - 1, n\}} (\mu^{m - 1}_i)^{n - \min\{m - 1, n\}}} \frac{h^2_\pi(\mu^m) \Delta(\mu^m)}{h_\pi^2(\mu^{m-1}) \Delta(\mu^{m-1})}e^{- \hpi_m(|\mu^m| - |\mu^{m-1}|)}  1_{\mu^{m - 1} \prec \mu^m} d\mu^m,
\end{multline*}
where we recall that $h^2_\pi(\mu)$ denotes the complex HCIZ integral.
\end{theorem}
\begin{remark}
Our statement of Theorem \ref{thm:gen-wish-dens} differs from that of \cite[Theorem 3.1]{DW} by a factor of 
\[
\frac{\prod_{i = 1}^{\min\{m, n\}} (\mu^m_i)^{n - \min\{m, n\}}}{\prod_{i = 1}^{\min\{m - 1, n\}} (\mu^{m - 1}_i)^{n - \min\{m - 1, n\}}}
\]
because each $\mu^m$ is padded with $0$'s to be length $n$ in \cite{DW}.
\end{remark}

\begin{corr} \label{corr:dw-corr1}
For $\beta = 2$, the eigenvalues $\mu^1 \prec \cdots \prec \mu^m$ of the multilevel generalized $\beta$-Wishart ensemble with parameters $(\pi, \hpi)$ have the law of the multivariate Bessel ensemble with parameters $(\pi, \hpi)$.  
\end{corr}
\begin{proof}
By Theorem \ref{thm:gen-wish-dens}, the joint law of $\{\mu^l_i\}$ is given by
\begin{equation} \label{eq:markov-dens}
\wW_{1, m, n} \prod_{j = 1}^m \prod_{i = 1}^n (\pi_i + \hpi_j) \prod_{i = 1}^{\min\{m, n\}} (\mu^m_i)^{n - \min\{m, n\}} h^2_\pi(\mu^m) \Delta(\mu^m) e^{-\sum_{j = 1}^m \hpi_j (|\mu^j| - |\mu^{j - 1}|)} 
\end{equation}
for a normalization constant $\wW_{1, m, n}$.  Now, by Lemma \ref{lem:hciz-complex}, we have that 
\[
h^2_\pi(\mu^m) = \BB_2^{n, n}(\pi, -\mu^m) = \BB_2^{\min\{m, n\}, n}(\mu^m, -\pi),
\]
meaning that (\ref{eq:markov-dens}) coincides with (\ref{eq:mvb-density}), as desired.
\end{proof}

\subsection{Matrix model at $\beta = 1$} 

In this section, we state and prove Theorem \ref{thm:beta1-wish} and Corollary \ref{corr:beta1-mb-ml}, which are analogues for $\beta = 1$ of Theorem \ref{thm:gen-wish-dens} and Corollary \ref{corr:dw-corr1}.  Our proof parallels the analysis of \cite{DW} in the $\beta = 2$ setting, but requires the identification of the real HCIZ integral with a multivariate Bessel function given in Section \ref{sec:mvb-hciz}.

\begin{theorem} \label{thm:beta1-wish}
For $\beta = 1$, the eigenvalues $\{\mu^m_i\}_{m \geq 1}$ of the multilevel generalized $\beta$-Wishart ensemble with parameters $(\pi, \hpi)$ form a Markov chain with transition kernel
\begin{multline*}
Q_{m - 1, m}^{\pi, \hpi}(\mu^{m-1}, d\mu^m) = \frac{\prod_{i = 1}^n (\pi_i + \hpi_m)^{\frac{1}{2}}}{\Gamma(m/2) \Gamma(1/2)^m} e^{-\frac{1}{2}\hpi_m(|\mu^m| - |\mu^{m-1}|)}\\
 \frac{\prod_{i = 1}^{\min\{m, n\}} (\mu^m_i)^{\frac{n - \min\{m, n\} - 1_{m \leq n}}{2}}}{\prod_{i = 1}^{\min\{m - 1, n\}} (\mu^{m-1}_i)^{\frac{n - \min\{m - 1, n\} - 1_{m \leq n}}{2}}} \frac{h^1_\pi(\mu^m)}{h^1_\pi(\mu^{m-1})}\Delta(\mu^m) \Delta(\mu^m, \mu^{m - 1})^{-1/2} 1_{\mu^{m-1} \prec \mu^m} d\mu^m,
\end{multline*}
where we recall that $h^1_\pi(\mu)$ denotes the real HCIZ integral.
\end{theorem}

\begin{remark}
If $\pi_i = 1$ and $\hpi_j = 0$, Theorem \ref{thm:beta1-wish} implies that $Q_{m - 1, m}(\mu^{m - 1}, d\mu^m) := Q_{m - 1, m}^{(1, \ldots, 1), (0, \ldots, 0)}(\mu^{m - 1}, d\mu^m)$ is given by
\begin{multline*}
Q_{m - 1, m}(\mu^{m-1}, d\mu^m) = \frac{ e^{-\frac{1}{2}(|\mu^m| - |\mu^{m-1}|)}}{\Gamma(m/2) \Gamma(1/2)^m} \frac{\prod_{i = 1}^m (\mu^m_i)^{\frac{n - \min\{m, n\} - 1_{m \leq n}}{2}}}{\prod_{i = 1}^{m - 1} (\mu^{m-1}_i)^{\frac{n - \min\{m - 1, n\} - 1_{m \leq n}}{2}}}\\
 \Delta(\mu^m) \Delta(\mu^m, \mu^{m - 1})^{-1/2} 1_{\mu^{m-1} \prec \mu^m} d\mu^m,
\end{multline*}
which agrees with the transition kernel of the ordinary multilevel $\beta$-Wishart ensemble given in \cite[Corollary 3]{FR}.
\end{remark}

\begin{corr} \label{corr:beta1-mb-ml}
For $\beta = 1$, the eigenvalues $\mu^1 \prec \cdots \prec \mu^m$ of the multilevel $\beta$-Wishart ensemble with parameters $(\pi, \hpi)$ have the law of the multivariate Bessel ensemble with parameters $(\pi, \hpi)$.
\end{corr}
\begin{proof}
By Theorem \ref{thm:beta1-wish}, we find that the joint density of $\{\mu^l_i\}$ is proportional to 
\begin{multline} \label{eq:markov-dens2}
\prod_{i = 1}^n \prod_{j = 1}^m (\pi_i + \hpi_j)^{\frac{1}{2}} e^{-\frac{1}{2}\sum_{j = 1}^m \hpi_j (|\mu^j| - |\mu^{j_1}|)}\\ \prod_{i = 1}^{\min\{m, n\}} (\mu^{\min\{m, n\}}_i)^{\frac{n - \min\{m, n\} - 1}{2}} h^1_\pi(\mu^m) \prod_{j = 1}^m \Delta(\mu^j) \Delta(\mu^j, \mu^{j-1})^{-1/2},
\end{multline}
By Lemma \ref{lem:hciz-real}, we have that $h^1_\pi(\mu^m) = \BB^{n, n}_1(\pi, -\mu^m/2)$ is proportional to
\[
\BB^{\min\{m, n\}, n}_1(\mu^m, -\pi/2) = \Gamma(1/2)^n \prod_{i = 1}^{\min\{m, n\}} (\mu^m_i)^{1/2} \wB^{\min\{m, n\}, n}_1(\mu^m, -\pi/2),
\]
meaning that (\ref{eq:markov-dens2}) agrees with (\ref{eq:mvb-density}), as desired.
\end{proof}

\begin{proof}[Proof of Theorem \ref{thm:beta1-wish}]
First, the result holds for $\pi = (1, \ldots, 1)$ and $\hpi = (0, \ldots, 0)$ by \cite[Corollary 3]{FR}.  Let $P^{\pi, \hpi}$ and $P^{\pi, \hpi}_M$ denote the distributions of $\mu^1, \ldots, \mu^m$ and $M_1, \ldots, M_m$ with parameters $(\pi, \hpi)$, and define $P := P^{(1, \ldots, 1), (0, \ldots, 0)}$ and $P_M := P^{(1, \ldots, 1), (0, \ldots, 0)}_M$.  Let $a_m$ denote the $m^\text{th}$ row of $A$, so that 
\[
M_{m} = M_{m-1} + a_m^T a_m \overset{d} = M_{m - 1} + (\pi + \hpi_m I)^{-1/2} z_m^T z_m(\pi + \hpi_m I)^{-1/2},
\]
where $z_m = a_m (\pi + \hpi_m I)$ is a $1 \times n$ vector of real standard Gaussian random variables.  The key step is the following Lemma \ref{lem:law-cont} parallel to \cite[Proposition 2.1]{DW} giving a change of measure between $P^{\pi, \hpi}$ and $P$.  

\begin{lemma} \label{lem:law-cont}
For $m \geq 1$, the $P^{\pi, \hpi}_M$-law of $M_{m} - M_{m - 1}$ is absolutely continuous with respect to the $P_M$-law of $M_{m} - M_{m - 1}$ with Radon-Nikodyn derivative
\[
\frac{dP^{\pi, \hpi}_M}{dP_M} = \prod_{i = 1}^n (\pi_i + \hpi_m)^{1/2} \exp\Big(-\frac{1}{2}(\hpi_m - 1) \Tr(M_m - M_{m-1}) -\frac{1}{2} \sum_{i = 1}^n \pi_i(M_{m, ii} - M_{m - 1, ii})\Big).
\]
\end{lemma}
\begin{proof}
The difference $M_m - M_{m - 1} = a_m^T a_m$ is determined by its diagonal entries, which have exponential distribution with parameter determined by $\pi$ and $\hpi$.  Comparing densities for $\pi, \hpi$ and $(1, \ldots, 1), (0, \ldots, 0)$ yields the conclusion.
\end{proof}

Applying Lemma \ref{lem:law-cont}, we find that
\begin{align*}
\frac{dP^{\pi, \hpi}}{dP}(\mu^1, \ldots, \mu^m) &= \EE_{P_M}\left[\frac{dP^{\pi, \hpi}_M}{dP_M}(M_1, \ldots, M_m) \mid \mu^1, \ldots, \mu^m\right]\\
&= \prod_{i = 1}^n \prod_{j = 1}^m (\pi_i + \hpi_j)^{1/2} e^{-\frac{1}{2} \sum_{m = 1}^m \hpi_m (|\mu^m| - |\mu^{m-1}|)} e^{\frac{1}{2} |\mu^m|} \EE_{P_M}\left[e^{-\frac{1}{2} \Tr(\pi M_m)} \mid \mu^1, \ldots, \mu^m\right].
\end{align*}
Because the $P_M$-distribution of $(M_1, \ldots, M_m)$ conditional on $(\mu^1, \ldots, \mu^m)$ is invariant under simultaneous conjugation by an element of an orthogonal group, we obtain by applying Lemma \ref{lem:hciz-real} that 
\[
\EE_{P_M}\left[e^{-\frac{1}{2} \Tr(\pi M_m)} \mid \mu^1, \ldots, \mu^m\right] = \EE_{P_M}\left[h^1_\pi(\mu^m) \mid \mu^1, \ldots, \mu^m\right] = h^1_\pi(\mu^m)
\]
and therefore that
\[
\frac{dP^{\pi, \hpi}}{dP}(\mu^1, \ldots, \mu^m) = \prod_{i = 1}^n \prod_{j = 1}^m (\pi_i + \hpi_j)^{1/2} e^{-\frac{1}{2} \sum_{m = 1}^m \hpi_m (|\mu^m| - |\mu^{m-1}|)} e^{\frac{1}{2} |\mu^m|} h^1_\pi(\mu^m).
\]
Since $\mu^1 \prec \mu^2 \prec \cdots$ is Markov under $P$ with transition kernel $Q_{m - 1, m}(\mu^{m - 1}, d\mu^m)$, we conclude that under $P^{\pi, \hpi}$ it is Markov with the desired transition kernel
\begin{align*}
Q^{\pi, \hpi}_{m - 1, m}(\mu^{m - 1}, d\mu^m) &= \frac{\frac{dP^{\pi, \hpi}}{dP}(\mu^1, \ldots, \mu^m)}{\frac{dP^{\pi, \hpi}}{dP}(\mu^1, \ldots, \mu^{m - 1})} Q_{m - 1, m}(\mu^{m - 1}, d\mu^m)\\
&= \prod_{i = 1}^n (\pi_i + \hpi_m)^{1/2} e^{-\frac{1}{2} (\hpi_m - 1)(|\mu^m| - \mu^{m - 1}|)} \frac{h^1_\pi(\mu^m)}{h^1_\pi(\mu^{m - 1})} Q_{m - 1, m}(\mu^{m - 1}, d\mu^m)\\
&= \frac{\prod_{i = 1}^n (\pi_i + \hpi_m)^{1/2}}{\Gamma(m/2)\Gamma(1/2)^m}  e^{-\frac{1}{2} \hpi_m(|\mu^m| - \mu^{m - 1}|)}\frac{\prod_{i = 1}^m (\mu^m_i)^{\frac{n - \min\{m, n\} - 1_{m \leq n}}{2}}}{\prod_{i = 1}^{m - 1} (\mu^{m-1}_i)^{\frac{n - \min\{m - 1, n\} - 1_{m \leq n}}{2}}}  \frac{h^1_\pi(\mu^m)}{h^1_\pi(\mu^{m - 1})}\\
&\phantom{=====}  \Delta(\mu^m) \Delta(\mu^m, \mu^{m - 1})^{-1/2} 1_{\mu^{m-1} \prec \mu^m} d\mu^m. \qedhere
\end{align*}
\end{proof}

\section{The Heckman-Opdam ensemble and the $\beta$-Jacobi ensemble} \label{sec:ho-jac}

In this section, we define the Heckman-Opdam and $\beta$-Jacobi ensembles and prove that at $\beta = 1, 2$, the $\beta$-Jacobi ensemble gives a matrix model for the principally specialized Heckman-Opdam ensemble.

\subsection{Definition of the Heckman-Opdam ensemble}

The Heckman-Opdam ensemble with parameters $\{\pi_i\}_{1 \leq i \leq n}$ and $\{\hpi_i\}_{i \geq 1}$ is the process on $\{\mu^l_i\}_{1 \leq i \leq \min\{n, l\}, 1\leq l \leq m}$ with joint density supported on 
\[
\mu^1 \prec \cdots \prec \mu^m \qquad \mu^l_i \in [0, \infty)
\]
and given by
\begin{multline} \label{eq:ho-density}
\wH_{\theta, m, n} \prod_{i = 1}^m \prod_{j = 1}^n \frac{\Gamma(\theta + \theta \pi_j + \theta\hpi_i)}{\Gamma(\theta \pi_j + \theta \hpi_i)} \Deltat(\mu^m)^{\theta} \prod_{i = 1}^{\min\{m, n\}} (1 - e^{-\mu_i^m})^{\theta(n - \min\{m, n\})}
 \prod_{i = 1}^{\min\{m, n\}} \frac{(1 - e^{-\mu_i^{\min\{m, n\}}})^{\theta - 1}}{(1 - e^{-\mu_i^m})^{\theta - 1}}\\ e^{-\theta\sum_{l = 1}^m \hpi_l (|\mu^l| - |\mu^{l-1}|)} \prod_{l = 1}^{m - 1} \frac{\Delta(e^{-\mu^l}, e^{-\mu^{l+1}})^{\theta - 1}}{\Delta(e^{-\mu^l})^{\theta - 1} \Delta(e^{-\mu^{l + 1}})^{\theta - 1}} e^{(\theta - 1) |\mu^l|} \wF_\beta^{\min\{m, n\}, n}(\mu^m, - \theta\pi)
\end{multline}
for a normalization constant $\wH_{\theta, m, n}$.  

\begin{prop} \label{prop:ho-valid}
The Heckman-Opdam ensemble with parameters $(\pi, \hpi)$ is a valid probability measure with level $m$ marginal density given by
\begin{multline} \label{eq:ho-marginal-density}
C^{\text{HO}}_{\theta, m, n} \prod_{i = 1}^m \prod_{j = 1}^n \frac{\Gamma(\theta + \theta \pi_j + \theta\hpi_i)}{\Gamma(\theta \pi_j + \theta \hpi_i)}
 \prod_{i = 1}^{\min\{m, n\}} (1 - e^{-\mu^m_i})^{\theta (\max\{m, n\} - \min\{m, n\})} \\ \Deltat(\mu^m)^{2\theta} \FF_\beta^{\min\{m, n\}, m}(\mu^m, -\theta\hpi) \wF_\beta^{\min\{m, n\}, n}(\mu^m, - \theta\pi)
\end{multline}
for a normalization constant $C^{\text{HO}}_{\theta, m, n}$.
\end{prop}
\begin{proof}
First, the density (\ref{eq:ho-density}) is non-negative by its definition.  By the defining integral formula (\ref{eq:ho-def}) for the multivariate Bessel function, integration over $\mu^1 \prec \cdots \prec \mu^{m - 1}$ in $\GT_{\mu^m}$ yields the expression (\ref{eq:ho-marginal-density}) with a non-zero constant of proportionality.  The Cauchy identity (\ref{prop:ho-cauchy}) for multivariate Bessel functions then ensures that (\ref{eq:ho-marginal-density}) and therefore (\ref{eq:ho-density}) has finite total mass and may therefore be normalized to a valid probability measure with the claimed marginal density.
\end{proof}

\subsection{Definition of the $\beta$-Jacobi ensemble} \label{sec:bj-def}

Let $\beta = 1, 2$, and let $X$ and $Y$ be infinite matrices of independent Gaussian random variables with mean $0$ and variance $1$ which are real if $\beta = 1$ and complex if $\beta = 2$.  Fix $m \leq n \leq A$, and let $X^{An}$ and $Y^{mn}$ denote the top left $A \times n$ and $m \times n$ corners of $X$ and $Y$.  The \textit{multilevel $\beta$-Jacobi ensemble} with parameters $(A, n)$ is the joint distribution of the matrices
\begin{equation} \label{eq:mat-model}
J_m := (X^{An})^* X^{An} ((X^{An})^* X^{An} + (Y^{mn})^* Y^{mn})^{-1}.
\end{equation}
Denote the $m$ smallest eigenvalues of $J_m$ in $[0, 1]$ by $\lambda_1^m, \ldots, \lambda_m^m$.  The \textit{$\beta$-Jacobi ensemble} is the joint density of the eigenvalues $\{\lambda^m_i\}_{1 \leq i \leq m, 1\leq m \leq n}$.  The following proposition shows that when conditioned on the value of $X^{An}$, the single level eigenvalues of the $\beta$-Jacobi ensemble have the distribution of the single level eigenvalues of a generalized $\beta$-Wishart ensemble.

\begin{prop} \label{prop:beta-jac-cond}
Let $\{\lambda^m_i\}$ be the multilevel $\beta$-Jacobi ensemble with parameters $(A, n)$.  When conditioned on the eigenvalues $\lambda_{X, i}$ of $(X^{An})^* X^{An}$, the joint distribution of $\tau^m_i := (\lambda^m_i)^{-1} - 1$ is given by the eigenvalues of the multilevel generalized $\beta$-Wishart ensemble with parameters $\pi_i = \lambda_{X, i}$ and $\hpi_j = 0$.
\end{prop}
\begin{proof}
Let $X := X^{An}$ and $Y_m := Y^{mn}$ and condition on the value of $X$.  We may find some $U_1$ in $U(A)$ for $\beta = 2$ and $O(A)$ for $\beta = 1$ so that
\[
X = U_1 \left(\begin{matrix} X_1 \\ 0 \end{matrix}\right),
\]
where $X_1$ is $n \times n$ upper-triangular and a.s. invertible.  For such a $U_1$, we have that 
\[
X^*X = X_1^* U_1^* U_1 X_1 = X_1^* X_1.
\]
The $m$ smallest eigenvalues $\{\lambda^m_i\}$ of $X^*X(X^*X + Y_m^*Y_m)^{-1}$ are the $m$ smallest solutions to 
\[
\det(\lambda (X^*X + Y_m^*Y_m) - X^*X) = 0 \qquad \iff \qquad \det((1/\lambda - 1) + X_1^{-1} X_1^{-*} Y_m^*Y_m) = 0.
\]
Notice now that $\tau^m_i := (\lambda^m_i)^{-1} - 1$ are the $m$ largest eigenvalues of $X_1^{-1}X_1^{-*} Y_m^* Y_m$ and further that $X_1^{-1} X_1^{-*} Y_m^*Y_m$ and $X_1^{-*}Y_m^* Y_m X_1^{-1} = (Y_m X_1^{-1})^*(Y_m X_1^{-1})$ are similar, hence have the same $m$ largest eigenvalues.  Because $A^*A$ and $AA^*$ have the same $m$ largest eigenvalues, this means that $\tau^m_i$ are the $m$ largest eigenvalues of $(Y_m X_1^{-1})(Y_m X_1^{-1})^* = Y_m X_1^{-1}X_1^{-*}Y_m^*$.  

Choose $V$ in $U(n)$ for $\beta = 2$ or $O(n)$ for $\beta = 1$ be a unitary/orthogonal matrix so that $X_1^{-1} X_1^{-*} = V\Lambda_X^{-1}V^*$ for a positive real diagonal matrix $\Lambda_X$.  The distribution of $Y_m$ is unitarily/orthogonally invariant, so we find that 
\[
Y_m X_1^{-1}X_1^{-*}Y_m^* = Y_mV \Lambda_X^{-1} V^*Y_m^* \overset{d} = Y_m \Lambda_X^{-1} Y_m^* = (Y_m \Lambda_X^{-1/2})(Y_m\Lambda_X^{-1/2})^*.
\]
This last product has the same $m$ largest eigenvalues as
\[
A_{X, m} := (Y_m \Lambda_X^{-1/2})^*(Y_m\Lambda_X^{-1/2}) = \Lambda_X^{-1/2} Y_m^* Y_m \Lambda_X^{-1/2},
\]
meaning that conditioned on $\lambda_X$, $\{\tau^m_i\}$ is equal in law to the eigenvalues of $A_{X, m}$.  On the other hand, we have the equality in law
\[
A_{X, m + 1} \overset{d} = A_{X, m} + (\lambda_{X, i}^{-1/2} z_i^*z_j \lambda_{X, j}^{-1/2}),
\]
where $z_i$ is a $1 \times n$ vector of standard unit Gaussian random variables, real for $\beta = 1$ and complex for $\beta = 2$.  Recognizing this as a way of generating the multilevel generalized $\beta$-Wishart distribution with $\pi_i = \lambda_{X, i}$ and $\hpi_j = 0$ yields the desired conclusion.
\end{proof}

\subsection{Matrix model at $\beta = 2$}

In this section we state and prove Theorem \ref{thm:bg-ver} and Corollary \ref{corr:bg-ver-ho} giving an explicit probability density for the multilevel $\beta$-Jacobi ensemble at $\beta = 2$ and identifying it with a principally specialized Heckman-Opdam ensemble.  Together these provide a link between the results of \cite{BG} and random matrix ensembles.

\begin{theorem} \label{thm:bg-ver}
For $\beta = 2$, the eigenvalues $\{\lambda^l_i\}$ of the first $m$ levels of the $\beta$-Jacobi ensemble with parameters $(A, n)$ are supported on interlacing sequences 
\[
\lambda^1 \prec \cdots \prec \lambda^m \qquad \lambda^l_i \in [0, 1]
\]
with joint density given by
\[
\wJ_{1, m, n} \Delta(\lambda^m) \prod_{i = 1}^m (\lambda^m_i)^{A + m - n - 1} (1 - \lambda^m_i)^{n - m}\prod_{l = 1}^{m - 1} \prod_{i = 1}^l (\lambda^l_i)^{-2}
\]
for a normalization constant $\wJ_{1, m, n}$.
\end{theorem}
\begin{proof}
Change variables to $\tau^m_i := (\lambda^m_{m + 1 - i})^{-1} - 1 \in [0, \infty)$.  Note that 
\[
d\lambda^m_{m + 1 - i} = - \frac{1}{(1 + \tau^m_i)^2} d\tau^m_i = -(\lambda^m_{m + 1 - i})^2 d\tau^m_i
\]
and that $\tau^m \prec \tau^{m+1}$ if and only if $\lambda^m \prec \lambda^{m + 1}$.  Therefore, it suffices for us to check that $\{\tau^m_i\}$ are supported on $\tau^1 \prec \cdots \prec \tau^m$ with measure given by
\begin{equation} \label{eq:bg-ver-ans}
p_m(\tau^m, \ldots, \tau^1) d\tau = C_m\, \Delta(\tau^m) \prod_{i = 1}^m (1 + \tau^m_i)^{-A - m} (\tau^m_i)^{n - m} \prod_{l = 1}^m \prod_{i = 1}^l d\tau^l_i
\end{equation}
for some constant $C_m$. 

We proceed by induction on $m$; for $m = 1$, this is an easy computation.  Suppose that $p_{m - 1}(\tau^{m - 1}, \ldots \tau^1)$ has the desired form.  Note that $X^*X$ has the law of a complex Wishart matrix of rank $A$ and level $n$, which means that its eigenvalues $\{\lambda_{X, i}\}$ have the density
\[
p_X(\lambda_X) = C_{2, A, n} \Delta(\lambda_X)^2 e^{- \sum_i \lambda_{X, i}} \prod_i \lambda_{X, i}^{A - n}
\]
for some constant $C_{2, A, n}$.  Further, by Proposition \ref{prop:beta-jac-cond}, conditioned on $\{\lambda_{X, i}\}$, the process $\{\tau^m_i\}$ has the law of the eigenvalues of a multilevel generalized $\beta$-Wishart ensemble with parameters $\pi_i = \lambda_{X, i}$ and $\hpi_j = 0$.  By Theorem \ref{thm:gen-wish-dens}, conditional on $\lambda_{X, i}$, the $\{\tau^m_i\}$ are Markovian with transition kernel
\[
Q^{\lambda_X}_{m - 1, m}(\tau^{m-1}, d\tau^m) = \prod_{i = 1}^n \lambda_{X, i} \frac{\prod_{i = 1}^m (\tau^m_i)^{n - m}}{\prod_{i = 1}^{m - 1}(\tau^{m - 1}_i)^{n - m + 1}} \frac{h^2_{\lambda_X}(\tau^m) \Delta(\tau^m)}{h^2_{\lambda_X}(\tau^{m-1})\Delta(\tau^{m-1})} 1_{\tau^{m - 1} \prec \tau^m} d\tau^m.
\]
In particular, their joint density is
\[
p^{\lambda_X}_m(\tau^m, \ldots, \tau^1) = \prod_{i = 1}^n \lambda_{X, i}^m \prod_{i = 1}^m (\tau_i^m)^{n - m} h^2_{\lambda_X}(\tau^m) \Delta(\tau^m) 1_{\tau^1 \prec \cdots \prec \tau^m}.
\]
By Bayes' rule, we have for $\tau^1 \prec \cdots \prec \tau^m$ that
\[
p_X(\lambda_X \mid \tau^{m - 1}, \ldots, \tau^1) = \frac{p_X(\lambda_X) p^{\lambda_X}_{m - 1}(\tau^{m-1}, \ldots, \tau^1)}{p_{m-1}(\tau^{m-1}, \ldots, \tau^1)}.
\]
Applying the inductive hypothesis and substituting in, the transition density of the unconditioned process is therefore
\begin{align*}
Q_{m - 1, m}&(\tau^{m - 1}, \ldots, \tau^1, d\tau^m) = \int_{\lambda_X} Q^{\lambda_X}_{m - 1, m}(\tau^{m-1}, d\tau^m) p_X(\lambda_X \mid \tau^{m-1}, \ldots, \tau^1) d\lambda_X\\
&= \int_{\lambda_X} \frac{p_X(\lambda_X) p^{\lambda_X}_m(\tau^m, \ldots, \tau^1)}{p_{m - 1}(\tau^{m - 1}, \ldots, \tau^1)} d\lambda_X d\tau^m\\
&= C_{2, A, n} C_{m - 1}^{-1} \int_{\lambda_X}h^2_{\lambda_X}(\tau^m) \Delta(\lambda_X)^2 e^{-\sum_i \lambda_{X_i}} \prod_i \lambda_{X, i}^{A - n + m} d\lambda_X \\
&\phantom{============} \frac{\prod_{i = 1}^m (\tau_i^m)^{n - m} \Delta(\tau^m)}{\prod_{i = 1}^{m - 1} (1 + \tau^{m - 1}_i)^{-A - m + 1} \prod_{i = 1}^{m - 1}(\tau^{m - 1}_i)^{n - m + 1} \Delta(\tau^{m - 1})}  d\tau^m\\
&= C_{2, A, n} C_{m - 1}^{-1} \int_{Z \sim (A + m) \times n} e^{-\Tr(Z^*Z \tau^m)} e^{-\frac{1}{2}||Z||^2} dZ\\
&\phantom{============} \frac{\prod_{i = 1}^m (\tau_i^m)^{n - m} \Delta(\tau^m)}{\prod_{i = 1}^{m - 1} (1 + \tau^{m - 1}_i)^{-A - m + 1}\prod_{i = 1}^{m - 1}(\tau^{m - 1}_i)^{n - m + 1}  \Delta(\tau^{m - 1})}  d\tau^m\\
&= C \frac{\prod_{i = 1}^{m - 1} (1 + \tau^{m-1}_i)^{A + m - 1}}{\prod_{i = 1}^m (1 + \tau^m_i)^{A + m}}\frac{\Delta(\tau^m)}{\Delta(\tau^{m-1})} \frac{\prod_{i = 1}^m (\tau^m_i)^{n - m}}{\prod_{i = 1}^{m - 1}(\tau^{m - 1}_i)^{n - m + 1}} d\tau^m
\end{align*}
for some constant $C$, where the integral in the fourth line is over the space of $(A + m) \times n$ complex-valued matrices, the third equality follows from the definition of the complex HCIZ integral, and the fourth equality follows from \cite[Theorem 6]{FR}.  This shows that the level-to-level transitions are Markov and that the joint density takes the claimed form (\ref{eq:bg-ver-ans}).
\end{proof}

\begin{corr} \label{corr:bg-ver-ho}
For $\beta = 2$, the transformation $\mu^l_i := - \log \lambda^l_i$ of the eigenvalues of the multilevel $\beta$-Jacobi ensemble have the law of the Heckman-Opdam ensemble with parameters $\pi = (A - n + 1, \ldots, A)$ and $\hpi = (0, 1, \ldots)$.
\end{corr}
\begin{proof}
By Theorem \ref{thm:bg-ver}, the joint density of $\lambda^1 \prec \cdots \prec \lambda^m$ is proportional to
\[
\Delta(\lambda^m) \prod_{i = 1}^m (\lambda^m_i)^{A + m - n - 1} (1 - \lambda^m_i)^{n - m}\prod_{l = 1}^{m - 1} \prod_{i = 1}^l (\lambda^l_i)^{-2}
\]
for $\lambda^l_i$ supported in $[0, 1]$.  Changing variables to $\mu_i^l = - \log\lambda^l_i$, the density of $\mu^1 \prec \cdots \prec \mu^m$ is proportional to
\[
\Delta(e^{-\mu^m}) e^{- (A + m - n)|\mu^m|} \prod_{i = 1}^m (1 - e^{-\mu^m_i})^{n - m} \prod_{l = 1}^{m - 1} e^{|\mu^l|}.
\]
For $\pi = (A - n + 1, \ldots, A)$, note that 
\[
\wF_2^{m, n}(\mu^m, -\pi) = e^{-\frac{n - 1}{2} |\mu^m| - (A - n + 1)|\mu^m|},
\]
which implies the desired by substitution into (\ref{eq:ho-density}).
\end{proof}

\subsection{Matrix model at $\beta = 1$}

In this section we state and prove Theorem \ref{thm:bg-ver-beta1} and Corollary \ref{corr:bg-ver-ho-beta1} giving an explicit probability density for the multilevel $\beta$-Jacobi ensemble at $\beta = 1$ and identifying it with a principally specialized Heckman-Opdam ensemble.

\begin{theorem} \label{thm:bg-ver-beta1}
For $\beta = 1$, the eigenvalues $\{\lambda^l_i\}$ of the first $m$ levels of the $\beta$-Jacobi ensemble with parameters $(A, n)$ are supported on interlacing sequences
\[
\lambda^1 \prec \cdots \prec \lambda^m \qquad \lambda^l_i \in [0, 1]
\]
with joint density given by
\[
\wJ_{1/2, m, n} \Delta(\lambda^m) \prod_{i = 1}^m (\lambda^m_i)^{\frac{A + m - n - 4}{2}} (1 - \lambda^m_i)^{\frac{n - m + 1}{2}} \prod_{l = 1}^{m - 1} \prod_{i = 1}^l (\lambda^l_i)^{-1} \Delta(\lambda^l) \Delta(\lambda^l, \lambda^{l + 1})^{-1/2}
\]
for a normalization constant $\wJ_{1/2, m, n}$.
\end{theorem}
\begin{proof}
As in the proof of Theorem \ref{thm:bg-ver}, change variables to $\tau^m_i := (\lambda^m_{m + 1 - i})^{-1} - 1 \in [0, \infty)$.  We wish to check that $\{\tau^m_i\}$ are supported on $\{\tau^1 \prec \cdots \prec \tau^m\}$ with density given by
\[ \label{eq:bg-beta1-des}
p_m(\tau^m, \ldots, \tau^1) = C_m\Delta(\tau^m) \prod_{i = 1}^m (1 + \tau^m_i)^{-\frac{A + m}{2}} (\tau^m_i)^{\frac{n - m + 1}{2}} \prod_{l = 1}^{m - 1}  \Delta(\tau^l) \Delta(\tau^l, \tau^{l + 1})^{-1/2}
\]
for some normalization constant $C_m$. 

We proceed by induction on $m$; for $m = 1$, the result follows from direct computation.  Suppose that $p_{m - 1}(\tau^{m - 1}, \ldots, \tau^{1}$ has the desired form.  Note that $X^*X$ has the law of a real Wishart matrix of rank $A$ and level $n$, which means that its eigenvalues $\{\lambda_{X, i}\}$ have the density
\[
p_X(\lambda_X) = C_{1, A, n} \Delta(\lambda_X) e^{-\frac{1}{2} \sum_i \lambda_{X, i}} \prod_i \lambda_{X, i}^{\frac{A - n - 1}{2}}
\]
for some constant $C_{1, A, n}$.  By Proposition \ref{prop:beta-jac-cond}, conditioned on $\{\lambda_{X, i}\}$, the process $\{\tau^m_i\}$ has the law of the eigenvalues of a multilevel generalized $\beta$-Wishart ensemble with parameters $\pi_i = \lambda_{X, i}$ and $\hpi_j = 0$.  By Theorem \ref{thm:beta1-wish}, conditional on $\{\lambda_{X, i}\}$, the $\{\tau^m_i\}$ are Markovian with some transition kernel $Q^{\lambda_X}_{m - 1, m}(\tau^{m - 1}, d\tau^m)$ and joint density 
\[
p^{\lambda_X}_m(\tau^m, \ldots, \tau^1) = \frac{\prod_{i = 1}^n (\lambda_{X, i})^{m/2}}{\Gamma(m/2) \cdots \Gamma(1/2) \Gamma(1/2)^{m (m + 1)/2}} \prod_{i = 1}^m (\tau^m_i)^{\frac{n - m - 1}{2}} h^1_{\lambda_X}(\tau^m) \prod_{l = 1}^m \Delta(\tau^l) \Delta(\tau^l, \tau^{l - 1})^{-1/2}.
\]
Therefore, we may compute the joint density of the unconditioned process to be
\begin{align*}
p_m(\tau^m, \ldots, \tau^1) &= \int_{\lambda_X} p_X(\lambda_X) p^{\lambda_X}_m(\tau^m, \ldots, \tau^1) d\lambda_X \\
&= \frac{C_{1, A, n}}{\Gamma(m/2) \cdots \Gamma(1/2) \Gamma(1/2)^{m(m + 1)/2}} \int_{\lambda_X} h^1_{\lambda_X}(\tau^m) \Delta(\lambda_X) e^{-\frac{1}{2}\sum_i \lambda_{X, i}} \prod_{i = 1}^n \lambda_{X, i}^{\frac{A + m - n -1}{2}} d\lambda_X \\
&\phantom{====} \prod_{i = 1}^m (\tau^m_i)^{\frac{n - m - 1}{2}} \prod_{l = 1}^m \Delta(\tau^l) \Delta(\tau^l, \tau^{l - 1})^{-1/2}.
\end{align*}
Notice now that 
\begin{align*}
\int_{\lambda_X} h^1_{\lambda_X}(\tau^m) \Delta(\lambda_X) e^{-\frac{1}{2}\sum_i \lambda_{X, i}} \prod_{i = 1}^n \lambda_{X, i}^{\frac{A + m - n -1}{2}} d\lambda_X &= \int_{Z \sim (A + m) \times n} e^{-\frac{1}{2} \Tr(Z^TZ \tau^m)} e^{-\frac{1}{2}||Z||^2} dZ\\
&= \int_{Z_{ij}} e^{-\frac{1}{2} \sum_{i = 1}^{A + m} \sum_{j = 1}^m (\tau^m_j + 1) Z_{ij}^2} d_{Z_{ij}}\\
&= \prod_{j = 1}^m (1 + \tau^m_j)^{-\frac{A + m}{2}}.
\end{align*}
Substituting back in, we obtain that 
\[
p_m(\tau^m, \ldots, \tau^1) = C \prod_{i = 1}^m (\tau^m_i)^{\frac{n - m - 1}{2}} (1 + \tau^m_i)^{-\frac{A + m}{2}} \prod_{l = 1}^m \Delta(\tau^l) \Delta(\tau^l, \tau^{l - 1})^{-1/2}
\]
for some constant $C$, as desired.
\end{proof}

\begin{corr} \label{corr:bg-ver-ho-beta1}
For $\beta = 1$, the transformation $\mu^l_i := - \log \lambda^l_i$ of the eigenvalues of the multilevel $\beta$-Jacobi ensemble have the law of the Heckman-Opdam ensemble with parameters $\pi = (A - n + 1, \ldots, A)$ and $\hpi = (0, 1, \ldots)$.
\end{corr}
\begin{proof}
In this case, by Theorem \ref{thm:bg-ver-beta1}, the joint density of $\lambda^1 \prec \cdots \prec \lambda^m$ is proportional to
\[
\Delta(\lambda^m) \prod_{i = 1}^m (\lambda^m_i)^{\frac{A + m - n - 4}{2}} (1 - \lambda^m_i)^{\frac{n - m + 1}{2}} \prod_{l = 1}^{m - 1} \prod_{i = 1}^l (\lambda^l_i)^{-1} \Delta(\lambda^l) \Delta(\lambda^l, \lambda^{l + 1})^{-1/2}
\]
for $\lambda^l_i$ supported in $[0, 1]$.  Changing variables to $\mu^l_i = - \log \lambda^l_i$, the density of $\mu^1 \prec \cdots \prec \mu^m$ is proportional to
\[
\Delta(e^{-\mu^m}) e^{- \frac{A + m - n - 2}{2}|\mu^m|} \prod_{i = 1}^m (1 - e^{-\mu^m_i})^{\frac{n - m + 1}{2}} \prod_{l = 1}^{m - 1} \Delta(e^{-\mu^l}) \Delta(e^{-\mu^l}, e^{-\mu^{l + 1}})^{-1/2}.
\]
Note that for $\pi = (A - n + 1, \ldots, A)$, we have 
\[
\wF^{m, n}_1(\mu^m, -\pi/2) = \Gamma(1/2)^{-m} e^{-\frac{m - 1}{4}|\mu^m| - \frac{A - n + 1}{2}|\mu^m|} \prod_{i = 1}^m (1 - e^{-\mu^m_i})^{-1/2},
\]
which implies the desired by substitution into (\ref{eq:ho-density}).
\end{proof}

\appendix

\section{Elementary asymptotic relations} \label{sec:asymp}

In this section we collect some elementary results on limits of various special functions.

\begin{lemma}[{\cite[Lemma 2.4]{BG}}] \label{lem:qpoch-asymp}
For $a, b \in \CC$ and $u(q)$ a complex-valued function defined in a neighborhood of $1$ so that $\lim_{q \to 1} u(q) = u$ with $0 < u < 1$, we have that 
\[
\lim_{q \to 1} \frac{(q^a u(q); q)}{(q^b u(q); q)} = (1 - u)^{b - a}.
\]
\end{lemma}

\begin{lemma} \label{lem:fin-qpoch-asymp}
For any $a, b$, $u(q)$ a complex-valued function defined in a neighborhood of $1$ so that $\lim_{q \to 1} u(q) = u$ with $0 < u < 1$, and $m(q)$ with $\lim_{q \to 1} q^{m(q)} = e^m$, we have that 
\[
\lim_{q \to 1} \frac{(q^au(q); q)_{m(q)}}{(q^b u(q); q)_{m(q)}} = \frac{(1 - u)^{b - a}}{(1 - u e^m)^{b - a}}.
\]
\end{lemma}
\begin{proof}
By Lemma \ref{lem:qpoch-asymp}, we obtain
\[
\lim_{q \to 1} \frac{(q^au(q); q)}{(q^a q^{m(q)}u(q); q)} \frac{(q^b q^{m(q)}u(q);q)}{(q^b u(q); q)} = \frac{(1 - u)^{b - a}}{(1 - u e^m)^{b - a}}. \qedhere
\]
\end{proof}

\begin{lemma}[{\cite[Corollary 10.3.4]{AAR}}] \label{lem:qgamma}
For $x \in \CC - \{0, -1, \ldots\}$, we have 
\[
\lim_{q \to 1^-} (1 - q)^{1 - x} \frac{(q; q)}{(q^x; q)} = \Gamma(x).
\]
\end{lemma}

\begin{lemma} \label{lem:ratio-qpoch}
We have
\[
\lim_{q \to 1} \frac{(q^a; q)}{(q^b; q)} (1 - q)^{a - b} = \frac{\Gamma(b)}{\Gamma(a)}.
\]
\end{lemma}
\begin{proof}
This follows by applying Lemma \ref{lem:qgamma} twice.
\end{proof}

\begin{lemma} \label{lem:gamma-rat}
We have 
\[
\lim_{\eps \to 0} \eps^{-\theta} \frac{\Gamma(\eps^{-1} a)}{\Gamma(\eps^{-1} a + \theta)} = a^{-\theta}.
\]
\end{lemma}
\begin{proof}
By Stirling's approximation, we have $\Gamma(z) = z^{z - 1/2} e^{-z} \sqrt{2\pi}(1 + O(z^{-1})$, hence
\[
\lim_{\eps \to 0} \eps^{-\theta} \frac{\Gamma(\eps^{-1}a)}{\Gamma(\eps^{-1}a + \theta)} = \lim_{\eps \to 0} \frac{a^{a\eps^{-1} - 1/2} \eps^{-a\eps^{-1} + 1/2} e^{-a\eps^{-1}}}{a^{a\eps^{-1} + \theta - 1/2} \eps^\theta (\eps^{-1} + \theta a^{-1})^{a \eps^{-1} - 1/2 + \theta} e^{-a \eps^{-1} - \theta}} = a^{-\theta}. \qedhere
\]
\end{proof}

\begin{lemma} \label{lem:f-lim}
For $f(u) = \frac{(tu;q)}{(qu;q)}$, $q = e^{-\eps}$, and $t = e^{-\theta \eps}$, we have 
\begin{itemize}
\item[(a)] $\lim_{\eps \to 0} f(u q^a) = (1 - u)^{1 - \theta}$

\item[(b)] $\lim_{\eps \to 0} f(q^a) \eps^{\theta - 1} = \frac{\Gamma(1 + a)}{\Gamma(\theta + a)}$.
\end{itemize}
\end{lemma}
\begin{proof}
Apply Lemma \ref{lem:qpoch-asymp} for (a) and Lemma \ref{lem:ratio-qpoch} for (b).
\end{proof}

\bibliographystyle{alpha}
\bibliography{sc-bib}
\end{document}